\documentclass[12pt]{amsart}
\usepackage{hyperref}
\usepackage{amsmath}
\usepackage{mathrsfs}
\usepackage{amssymb}
\usepackage{tikz-cd}
\usepackage{mathtools}
\calclayout
\DeclareMathOperator{\Spec}{Spec}
\newcommand{\PP}{\mathbb{P}}
\newcommand{\ZZ}{\mathbb{Z}}

\newcommand{\CC}{\mathbb{C}}
\newcommand{\QQ}{\mathbb{Q}}

\newcommand{\FF}{\mathbb{F}}

\numberwithin{equation}{section}

\newtheorem{theorem}[equation]{Theorem}

\theoremstyle{definition}

\newtheorem*{remark}{Remark}
\begin{document}
\title{Rational linear subspaces of hypersurfaces over finite fields}
\author[M. I. de Frutos-Fern\'andez]{Mar\'ia In\'es de Frutos-Fern\'andez}
\address{
Hausdorff Center for Mathematics
\\Universit\"at Bonn\\ 53115 Germany}
\email{midff@math.uni-bonn.de}
\author[S. Garai]{Sumita Garai}
\address{Perelman School of Medicine\\ University of Pennsylvania, PA 19104}
\email{sgarai@upenn.edu}
\author[K. Isham]{Kelly Isham}
\address{Department of Mathematics\\Colgate University, Hamilton, NY 13346}
\email{kisham@colgate.edu}
\author[T. Murayama]{Takumi Murayama}
\address{Department of Mathematics\\Purdue University\\West Lafayette, IN 47907}
\email{murayama@purdue.edu}
\author[G. Smith]{Geoffrey Smith}
\address{Department of Mathematics\\ University of Illinois Chicago\\ Chicago, IL 60607}
\email{geoff@uic.edu}
\date{}
\begin{abstract}
Let $X \subset \PP^n$ be a hypersurface of degree $d$ defined over a
finite field of characteristic $p > 0$.
We prove that if $n \ge r + \binom{d+r}{r+1}$, then $X$
contains a rational $r$-plane.
We prove better bounds when $X$ is smooth and $p$ is sufficiently large.
We also present experimental data regarding the existence of rational lines on
cubic threefolds over $\mathbb{F}_7$, $\mathbb{F}_8$, and $\mathbb{F}_9$.
In particular, we construct an example of a smooth cubic threefold over
$\mathbb{F}_7$ with exactly $8$ rational lines.
It remains an open question whether smooth cubic threefolds over $\mathbb{F}_7$,
$\mathbb{F}_8$, and $\mathbb{F}_9$ always contain a rational line.
\end{abstract}
\maketitle
\section{Introduction}
In 1849, Cayley \cite{Cay1849} and Salmon \cite{Sal1849} showed that smooth
cubic surfaces over the complex numbers contain exactly 27 lines.
This result does not extend to cubic surfaces over finite fields.
For example, the cubic surfaces
\[
  \bigl\{x_1^3+x_2^3+x_3^3+ax_4^3 = 0\bigr\} \subset \PP^3_{\FF_q}
\]
have no $\FF_q$-lines whenever $q \equiv 1 \pmod 3$ and $a \in \FF_q$ is not a
cube \cite[\S3]{DLR17}.
See \cite[\S5]{SD81} and \cite[\S1.5]{KSC04} for more examples.
On the other hand, Debarre, Laface, and Roulleau
\cite{DLR17} showed that in higher dimensions, cubic hypersurfaces $X \subset
\PP^n_{\FF_q}$ have $\FF_q$-rational lines for all $n \geq 4$, except possibly
when $n=4$ and $q \leq 9$ or $n=5$ and $q=3$.
They also constructed examples of cubic threefolds with no rational lines when
$q \le 5$.
\par A natural follow-up question is: \emph{When do hypersurfaces over finite
fields contain rational linear subspaces of higher dimensions?}
The case of rational points is well-understood because of the
Chevalley--Warning theorem \cite{Chev35,War35},
the Lang--Weil bound \cite{LW54},
the Ax--Katz theorem \cite{Ax64,Kat71}, etc.
As mentioned above, there are examples of cubic surfaces over finite fields
without any rational lines, and
the case of rational lines on cubic hypersurfaces of arbitrary dimension
was studied in \cite{DLR17}.
On the other hand, there is less known about the existence
of rational linear subspaces of dimensions \(\ge 2\).\medskip
\par In the first half of this paper, we show that if $X$ is a hypersurface over
a finite field whose dimension is sufficiently large compared to the degree
of $X$, then $X$ contains a rational subspace.
For smooth $X$, we have the following result:
\begin{theorem}\label{main}
Fix $d \ge 1$, $n \ge 4$, and $r \geq 1$.
Suppose that one of the following hypotheses hold:
\begin{enumerate}
  \item $r = 1$ and $n \ge 2d-1$.
  \item $r \geq 2$ and
    \(
      n\geq 2 \displaystyle\binom{d+r-1}{r}+r
    \).
\end{enumerate}
Then, there exists a finite set of primes $\mathcal{P}(d,n,r)$
depending only on $d$, $n$, and $r$ such
that the following holds:
Every smooth hypersurface $X$ of degree $d$ in $\PP^n$ over a finite field of
characteristic $p \notin \mathcal{P}(d,n,r)$ contains a rational $r$-plane.
\end{theorem}
For possibly singular $X$ but for \emph{arbitrary} characteristics $p >0$,
we have the following more restrictive bound on $d$.
This result generalizes \cite[Theorem 6.1]{DLR17}.
\begin{theorem}\label{secondResult}
Let $X$ be the zero locus of some degree $d$ polynomial in $\PP^n_k$ with $k$
quasi-algebraically closed, for example a finite field.
Let $r$ be a positive integer such that
\[
n\geq r+ \binom{d+r}{r+1}.
\]
Then, \(X\) contains a rational \(r\)-plane.
Moreover, for any rational $r'$-plane $\Lambda$ contained in $X$ with $0 \le r'<r$, there is a rational $r$-plane containing $\Lambda$ and contained in $X$.
\end{theorem}
Finite fields are quasi-algebraically closed by the Chevalley--Warning theorem \cite{Chev35,War35}.\medskip
\par In the second half of this paper, 
we explore the existence of rational lines on threefolds over $\FF_q$ for $q=7,8,9$.
It remains an open question whether smooth cubic threefolds over these fields
always contain a rational line.
While we were not able to resolve the existence question, we ran some computer experiments that we summarize in Section \ref{sect:experiment}. The statistics in Section \ref{sect:experiment} do not seem to clarify whether all threefolds over $\FF_7$ should contain at least one rational line; however, they suggest that all threefolds over $\FF_q$ for $q=8,9$ probably will.
In particular, we construct a cubic threefold over $\FF_7$ with exactly $8$
rational lines in \eqref{eq:8linesf7}.

\begin{remark}
Another recent work that addresses linear subspaces of hypersurfaces is
\cite{KP21}.
Therein, Kazhdan and Polishchuk show that if $X \subset \PP^n$ is a hypersurface
of degree $d$ defined over a perfect field $k$ and if the base change
$X_{\bar{k}}$ to the algebraic closure of $k$ contains a
linear subspace of codimension $e$ in $\PP^n$, then $X$ contains a $k$-rational linear
subspace of codimension $\le de$.
\end{remark}
\subsection*{Acknowledgements}
This work was undertaken while the authors were participants at the 2019 AMS Mathematics Research Community on Explicit Methods in Characteristic $p$, and we are grateful to the organizers and the AMS for support. In addition, we would like to thank Kiran Kedlaya,  Bjorn Poonen, and Eric Riedl for helpful conversations.
TM was supported by the National Science Foundation under Grant No.\ DMS-1902616.
\section{Proofs}
\subsection{Proof of Theorem \ref{main}}\label{sect:mainproof}
The key geometric input in the proof of Theorem \ref{main} will be the following theorem of Beheshti and Riedl.
For the statement below, recall that for a projective variety $X \subset \PP^n$,
the \emph{Fano variety of $r$-planes} $F_r(X)$ is the
subsceheme of the Grassmannian $G(r+1,n+1)$ parametrizing the $r$-planes
contained in $X$.
See \cite{AK77}.
\begin{theorem}[{Beheshti and Riedl \cite[Theorem 1.3]{BR19}}]\label{BR19}
Let $X\subset \PP^n_\CC$ be a smooth hypersurface of degree $d$ and suppose $n\geq 2d-1$ and $n \geq 4$. Then, the Fano variety of lines $F_1(X)$ is irreducible of the expected dimension $2n-d-3$.

Likewise, the Fano variety $F_r(X)$ of $r$-planes will be irreducible of the expected dimension $(r+1)(n-r)-\binom{d+r }{ r}$ provided
\[
n\geq 2 \binom{d+r-1 }{ r}+r.
\]
\end{theorem}
Given $n,d,r$, we will show this result holds over arbitrary fields of all but finitely many characteristics, allowing us to apply the Lang--Weil bound:
\begin{theorem}[{Lang and Weil \cite[Theorem 1]{LW54}}]\label{langweil}
  There exists a constant $A(n,d,r)$ depending only on $n,d,r$ such that the
  following holds:
  For every geometrically irreducible projective variety $V \subset
  \PP^n_{\FF_q}$ of degree $d$ and dimension $r$ over a finite field $\FF_q$, we
  have
  \[
    \Bigl\lvert \bigl\lvert V(\FF_q) \bigr\rvert - q^r \Bigr\rvert \le (d-1)(d-2)
    q^{r-\frac{1}{2}} + A(n,d,r)q^{r-1}.
  \]
\end{theorem}
Throughout the proof of Theorem \ref{main}, all schemes will be defined over
$\Spec(\ZZ)$ unless otherwise indicated.
\begin{proof}[Proof of Theorem \ref{main}]
  Let $\PP^N$ be the parameter space of all degree $d$ hypersurfaces in $\PP^n$,
  where $N=\binom{d+n}{d}-1$. Let $Y\subset\PP^N $ be the parameter space of all
  smooth hypersurfaces, and let $\pi\colon F\to Y$ be the universal Fano scheme of $r$-planes \cite[p.\ 5 and Theorem 3.3\((i)\)]{AK77}.
  This morphism is finitely presented, so by \cite[Th\'eor\`eme 9.7.7]{EGA4.3}, the locus
  \[
    W \coloneqq \bigl\{y \in Y \mathbin{:} F_y\ \text{is geometrically
    irreducible}\bigr\}
  \]
  is constructible. 
  By Theorem \ref{BR19}, we know that after base change to $\QQ$, we have $Y_\QQ
  \subset W_\QQ$.
  Thus, for some open set $U\subset Y$ such that $U_\QQ=Y_\QQ$, we have that
  every Fano scheme $F_y$ corresponding to $y\in U$ is geometrically
  irreducible. The complement $Z \coloneqq U\setminus Y$ is closed in $Y$, so
  its image under the final map $Z\to \Spec(\ZZ)$ is constructible and does not
  contain the generic point. Thus, the image of $Z$ in $\Spec(\ZZ)$ is a finite
  set of primes $p_1,\ldots,p_s$. 
  \par Set
  \[
    R\coloneqq\Spec\biggl(\ZZ\biggl[\frac{1}{p_1\cdots p_s}\biggr]\biggr),
  \]
  in which case the map $\pi\colon F_R\to Y_R$ is a finitely presented map with geometrically irreducible fibers.
  Applying Theorem \ref{BR19} and the upper semicontinuity of fiber dimension \cite[Th\'eor\`eme 13.1.3]{EGA4.3}, the fibers are all of dimension at least $(r+1)(n-r)-\binom{d+r}{r}$.
  Moreover, we have that $\pi\colon F_R\to Y_R$ is projective. Then, by the Lang--Weil
  bound (Theorem \ref{langweil}), for $q$ sufficiently large, we have that for
  any $\FF_q$ point $\{y\}\to Y$ the Fano scheme $F_y$ has a rational point.
\end{proof}

\subsection{Proof of Theorem \ref{secondResult}}
Over finite fields, the key ingredient for the proof of Theorem
\ref{secondResult} is the Chevalley--Warning theorem:
\begin{theorem}[{Chevalley \cite[Th\'eor\`eme on p.\ 75]{Chev35} and Warning
  \cite[p.\ 79]{War35}}]\label{chevwar}
  Let $k$ be a finite field of characteristic $p > 0$.
  Consider $s$ polynomials $f_i \in k[x_1,x_2,\ldots,x_n]$ of degree $d_i$ such
  that $n > \sum_{i=1}^s d_i$.
  Let $V = \{f_1=f_2=\cdots=f_s=0\} \subseteq \mathbb{A}^n_k$.
  Then,
  \[
    \bigl\lvert V(k)\bigr\rvert \equiv 0 \pmod p.
  \]
\end{theorem}
Since $p \ge 2$, the Chevalley--Warning theorem (Theorem \ref{chevwar})
implies that finite fields are
quasi-algebraically closed.
It therefore suffices to prove Theorem \ref{secondResult} for
quasi-algebraically closed fields.
\begin{proof}[Proof of Theorem \ref{secondResult}]
  We first note that \(X(k)\) is nonempty by 
  the assumption that $k$ is quasi-algebraically closed and the fact that
  \[
    n \ge r + \binom{d+r}{r+1} \ge r + d > d.
  \]
  By setting \(r' = 0\), it therefore suffices to show the ``Moreover'' claim.
  Additionally,
  we may assume $r'=r-1$, as for more general $r'$ we may simply repeatedly apply the result to produce first an $(r'+1)$-plane contained in $X$, then an $(r'+2)$-plane, and so on.

  Say $X$ is cut out by $F$, and let $\Lambda$ be a rational $(r-1)$-plane
  contained in $X$. By applying some automorphism of $\PP^n$, we may assume $\Lambda$ is the plane cut out by $x_{r}=\cdots=x_n=0$. The space of $r$-planes in $\PP^n$ containing $\Lambda$ is naturally identified with the $(n-r)$-plane $\Lambda^\circ$ cut out by $x_0=\cdots=x_{r-1}=0$; an $r$-plane $\Lambda'$ containing $\Lambda$ is identified with its point of intersection with the plane $\Lambda^\circ$. Writing $d_\bullet=(d_0,\ldots, d_{r-1})$, we can rewrite $F$ in the form
  \[
    F(x_0,\ldots,x_n)=\sum_{d_\bullet}x_0^{d_0}\cdots x_{r-1}^{d_{r-1}}\,F_{d_\bullet}(x_r,\ldots,x_n),
  \]
  with each $F_{d_\bullet}$ homogeneous of degree $d-\sum_{i}d_i$. Because $F$ vanishes on $\Lambda$, $F_{d_\bullet}=0$ if $d-\sum_i d_i=0$, and the space of $r$-planes through $\Lambda$ in $X$ is cut out by
  \[
    \bigl\{F_{d_\bullet}\mathbin{:}0\leq d_0+\cdots+d_{r-1}< d\bigr\}
  \]
  on $\Lambda^\circ$. For $j$ any positive integer, there are $\binom{j+r-1}{r-1}$ ways to write $j$ as a sum $j=i_0+\cdots+i_{r-1}$ with each $j_\ell$ a nonnegative integer. Thus the sum of the degrees of all the $F_{d_\bullet}$ is given by 
  \[
    \sum_{0\leq j <d}\binom{j+r-1}{r-1}(d-j).
  \]
  Using the hockey-stick identity on this sum gives the sum of degrees
  \[
    \sum_{0\leq j <d}\binom{j+r-1}{r-1}(d-j)=\binom{d+r}{r+1}.
  \]
  Since $k$ is quasi-algebraically closed and we have
  \[
    n-r\geq \binom{d+r}{r+1}
  \]
  by hypothesis, the locus in $\PP^{n-r}$ cut out by the $F_I$ includes a rational point, giving a rational $r$-plane through $\Lambda$.
\end{proof}

\section{Experimental Data}\label{sect:experiment}
	In \cite{DLR17} the authors prove that any smooth cubic threefold $X \subset \mathbb{P}^4_{\mathbb{F}_q}$ contains a rational line whenever $q \geq 11$. They also present examples of smooth cubic threefolds with no lines defined over $\mathbb{F}_q$ for $q=2,3,4,5$.  This existence question remains open in the cases $q=7,8,9$. We run some computer experiments to collect data about these open cases.
	
	For each $q$ in $\{2,3,4,5,7,8,9,11\}$, we obtained a random sample of $10^4$ cubic threefolds defined over $\mathbb{F}_q$ and we computed the 
	number of rational lines in each threefold using Magma \cite{magma}. We repeated this experiment with random samples of $10^4$ smooth cubic threefolds. The sample size was restricted by the exponential growth of the computation time: for instance, running the code on a $10^4$ sample of cubic threefolds took approximately four hours and 42 minutes of CPU time for $q=7$, and three days and 20 hours for
	$q=11$. The results are summarized in Tables \ref{tab:3fold} and \ref{tab:3fold_smooth}, respectively.
	
	\begin{table}[ht!]
		\centering
		\renewcommand{\arraystretch}{1.1}
		\begin{tabular}{|c|c|c|c|c|c|c|}
			\hline
			$q$ & Min & Max & Median & Mean & SD \\
			\hline
			2 & 0 & 46 & 9 & 9.697 & 5.835\\
			3 & 0 & 71 & 13 & 14.966 & 8.517\\
			4 & 0 & 115 & 21 & 22.710 & 11.603\\
			5 & 4 & 145 & 30 & 32.390 & 14.967\\	
			7 & 11 & 224 & 55 & 58.332 & 22.759\\
			8 & 16 & 248 & 71 & 74.717 & 27.071\\
			9 & 25 & 314 & 87 & 92.066 & 32.206\\
			11 & 35 & 357 & 129 & 134.154 & 42.089\\
			
			\hline
		\end{tabular}%
		\renewcommand{\arraystretch}{1}
		\vspace{5pt}
		\caption{Number of lines on cubic threefolds over $\mathbb{F}_q$.}
        \label{tab:3fold}
	\end{table}%
	\begin{table}[ht!]
		\centering
		\renewcommand{\arraystretch}{1.1}
		\begin{tabular}{|c|c|c|c|c|c|c|}
			\hline
			$q$ & Min & Max & Median & Mean & SD \\
			\hline
			2 & 0 & 28 & 6 & 6.9778 & 4.482\\
			3 & 0 & 66 & 12 & 13.0187 & 7.510\\
			4 & 1 & 92 & 19 & 20.9622 & 10.694\\
			5 & 3 & 126 & 29 & 31.0481 & 14.322\\
			7 & 8 & 209 & 53 & 56.8357 & 22.173\\
			8 & 15 & 210 & 69 & 72.8802 & 27.003\\
			9 & 16 & 279 & 87 & 91.1203 & 31.266\\
			11 & 43 & 376 & 126 & 131.633 & 40.552\\
			
			\hline
		\end{tabular}%
		\renewcommand{\arraystretch}{1}
		\vspace{5pt}
		\caption{Number of lines on smooth cubic threefolds over $\mathbb{F}_q$.}
        \label{tab:3fold_smooth}
	\end{table}%

	\begin{figure}[t]
		\begin{center}
			\begin{minipage}[t]{.49\textwidth}				
				\begin{center}
					\includegraphics[width=1.0\textwidth]{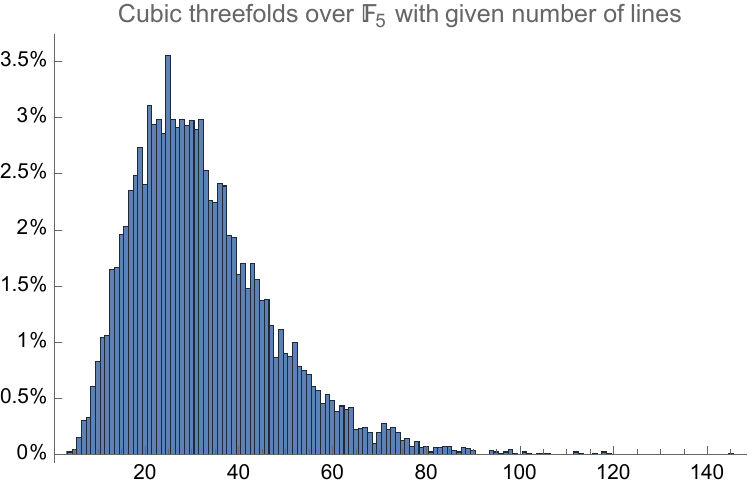}
				\end{center}
			\end{minipage}
			\begin{minipage}[t]{.49\textwidth}				
				\begin{center}
					\includegraphics[width=1.0\textwidth]{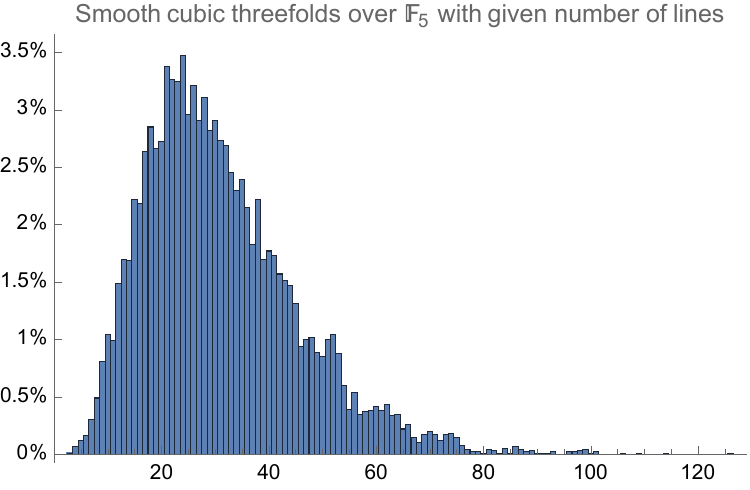}
				\end{center}
			\end{minipage}
			
			\begin{minipage}[t]{.49\textwidth}
				\begin{center}
					\includegraphics[width=1.0\textwidth]{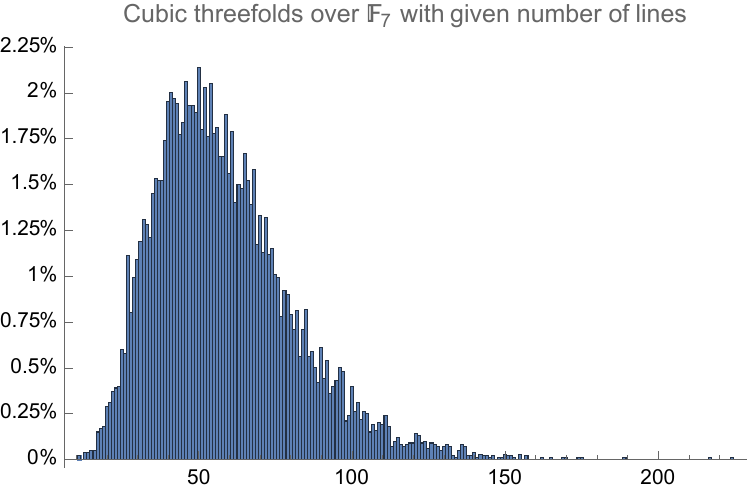}
				\end{center}
			\end{minipage}
			\begin{minipage}[t]{.49\textwidth}
				\begin{center}
					\includegraphics[width=1.0\textwidth]{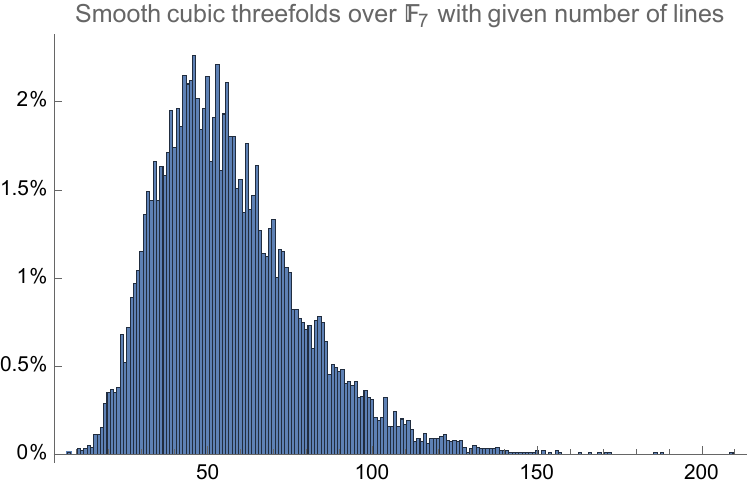}
				\end{center}
			\end{minipage}
			
				\begin{minipage}[t]{.49\textwidth}
					\begin{center}
						\includegraphics[width=1.0\textwidth]{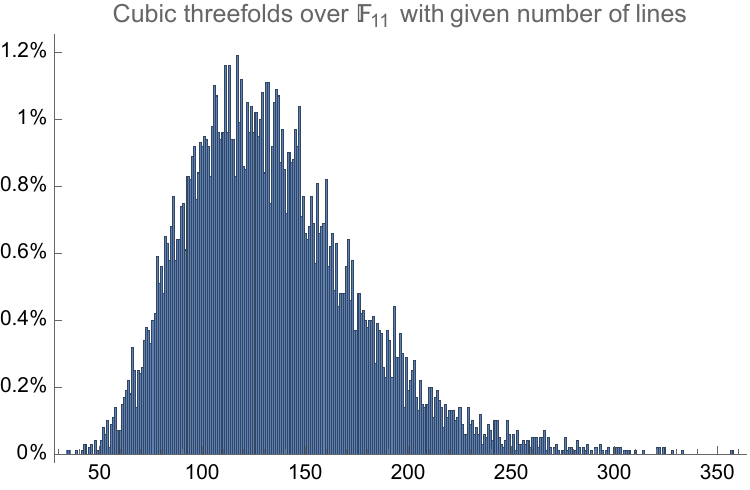}
					\end{center}
				\end{minipage}
				\begin{minipage}[t]{.49\textwidth}
					\begin{center}
						\includegraphics[width=1.0\textwidth]{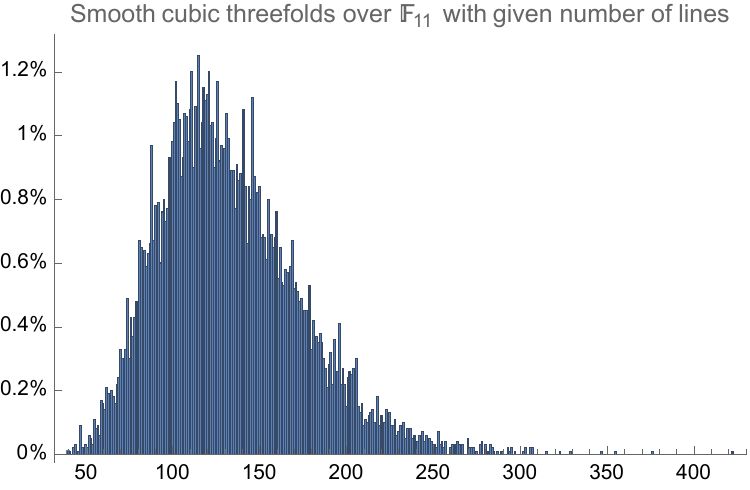}
					\end{center}
				\end{minipage}
		\end{center}
	\caption{Comparing number of lines for $q=5,7,11$.}
  \label{fig:histograms}
	\end{figure}
	
  We observe in Table \ref{tab:3fold} that the average number of lines in each sample of cubic threefolds is very close to the theoretical approximation from \cite[Formula (17)]{DLR17}.
  Denoting by \(G := G(2,n+1)\) the Grassmannian of lines in \(\mathbb{P}^n_{\mathbb{F}_q}\) and by \(\mathbb{P} := \mathbb{P}(H^0(\mathbb{P}^n_{\mathbb{F}_q},\mathcal{O}_{\mathbb{P}^n_{\mathbb{F}_q}}(d)))\) the parameter space of degree \(d\) hypersurfaces in \(\mathbb{P}^n_{\mathbb{F}_q}\), this approximation says that the average number of rational lines on a hypersurface \(X \subseteq \mathbb{P}^{n}_{\mathbb{F}_q}\) of degree \(d\) is
  \begin{equation}\label{eq:dlr24}
    \frac{\bigl\lvert G(\mathbb{F}_q) \bigr\rvert \bigl( q^{\dim(\mathbb{P}) - d} - 1 \bigr)}{\bigl( q^{\dim(\mathbb{P}) + 1} - 1 \bigr)} \sim 
    \bigl\lvert G(\mathbb{F}_q) \bigr\rvert  \bigl( q^{\dim(\mathbb{P}) - d - 1} \bigr).
  \end{equation}
  As stated in \cite[\S4.6]{DLR17},
  for cubic threefolds, the right hand side of \eqref{eq:dlr24} is
	\[ q^2 + q + 2 + 2q^{-1} + 2q^{-2} + q^{-3} + q^{-4} . \]
	The smooth cubic threefolds in the second set of samples contain slightly
  fewer rational lines on average, as recorded in Table \ref{tab:3fold_smooth}.
	
	We could not find examples with no rational lines for $q=7,8,9$. When $q=7$, we found an example of a smooth cubic threefold containing exactly 8 rational lines, given by
  \begin{equation}\label{eq:8linesf7}
    \begin{aligned}
      2x_1^3 &+ 6x_1^2x_2 + 5x_1x_2^2 + 5x_2^3 + 3x_1^2x_3 + 3x_1x_2x_3 + 
      4x_1x_3^2 + 4x_2x_3^2 + x_3^3 + x_1^2x_4 \\ &+ 4x_1x_2x_4 + x_2^2x_4 +5x_1x_3x_4 + 
      5x_2x_3x_4 + 5x_3^2x_4 + 2x_1x_4^2 + 6x_2x_4^2 \\ &+ 3x_3x_4^2 + 5x_4^3 + 2x_1^2x_5 
      + 2x_2^2x_5 + 2x_2x_3x_5 +6x_3^2x_5 +
      2x_1x_4x_5 \\ &+ 2x_2x_4x_5 + 4x_3x_4x_5 + 
      2x_4^2x_5 + 4x_1x_5^2 + 4x_2x_5^2 + 3x_3x_5^2 + x_4x_5^2 + x_5^3 = 0.
    \end{aligned}
  \end{equation}
	
  In Figure \ref{fig:histograms} we show the sample distributions we obtained for $q=5$ (for which there exist cubic threefolds with no rational lines), $q=7$ (for which the existence of rational lines remains open), and $q=11$ (for which it is known that all smooth cubic threefolds contain a rational line).

\bibliographystyle{plain}
\bibliography{hypersurfaces}
\end{document}